\documentclass[12pt]{amsart}
\usepackage[headings]{fullpage}

\usepackage{amsmath,amssymb,amsthm}
\usepackage{graphicx}
\usepackage{enumerate}
\usepackage{url}
\usepackage{slashed}
\usepackage{bm}
\usepackage[german,english]{babel}

\usepackage[bookmarks=true,%
    colorlinks=true,%
    linkcolor=blue,%
    citecolor=blue,%
    filecolor=blue,%
    menucolor=blue,%
    urlcolor=blue,%
    breaklinks=true]{hyperref}

\usepackage[all]{xy}
\usepackage{verbatim}

\newtheorem{thm}{Theorem}[]
\newtheorem*{thm*}{Theorem}
\newtheorem{lem}[thm]{Lemma}

\newtheorem{ques}[thm]{Question}
\newtheorem{claim}[thm]{Claim}

\newtheorem{rem}[thm]{Remark}

\newcommand{\co}{\colon\thinspace}

\newcommand{\param}{{\mathchoice{\mkern1mu\mbox{\raise2.2pt\hbox{$
\centerdot$}}
\mkern1mu}{\mkern1mu\mbox{\raise2.2pt\hbox{$\centerdot$}}\mkern1mu}{
\mkern1.5mu\centerdot\mkern1.5mu}{\mkern1.5mu\centerdot\mkern1.5mu}}}

\begin{document}
\title[Lifts of simple curves]{Lifts of simple curves in 
finite regular coverings of closed surfaces}
\author{Ingrid Irmer}
\address{Mathematics Department\\
Technion, Israel Institute of Technology\\
Haifa, 32000\\
Israel 
}
\email{ingridi@technion.ac.il}
\date{21 November, 2017}

\begin{abstract}
Suppose $S$ is a closed orientable surface and $\tilde{S}$ is a finite sheeted regular cover of $S$. When studying mapping class groups, the following question arose: Do the lifts of simple curves from $S$ generate $H_{1}(\tilde{S},\mathbb{Z})$? A family of examples is given for which the answer is ``no''.
\end{abstract}

\maketitle

{\footnotesize
\tableofcontents
}

\section{Introduction}
\label{sect:intro}

Let $S$ be a genus $g$ closed, orientable surface with base point and no boundary. Fix $p\co \tilde{S} \to S$, a finite-sheeted regular covering
of $S$, where $\tilde S$ is connected. The \emph{simple curve homology of $p$} (denoted by $sc_{p}(H_{1}(\tilde{S};\mathbb{Z})$) is the span of $[\tilde{\gamma}]$ in $H_1(\tilde{S};\mathbb{Z})$ such that $\tilde{\gamma}$ is a connected component of $p^{-1}(\gamma)$ and $\gamma$ a simple closed curve in $S$.\\

Recall that the Torelli group of a surface is the subgroup of the 
mapping class group that consists of surface diffeomorphisms that act
trivially on homology with integer coefficients. 
A survey of the Torelli group can be found in \cite{Johnson}. 
The following question was posed by Julien March\'{e} on Mathoverflow, \cite{MO}, and arose while studying the ergodicity of the action of the Torelli group on 
$\mathrm{SU}(2)$-character varieties of surfaces, \cite{Marche}:

\begin{ques}[see \cite{MO}]
\label{Mainquestion}
\begin{enumerate}
\item Does $sc_{p}(H_{1}(\tilde{S};\mathbb{Z})) =H_1(\tilde{S},\mathbb{Z})$?
\item If not, how can we characterize the submodule $sc_{p}(H_{1}(\tilde{S};\mathbb{Z}))$?
\end{enumerate}
\end{ques}

Since writing the first version of this paper, there has been much progress on answering Question \ref{Mainquestion} and related questions. By using quantum $\mathrm{SO}(3)$-representations coming from TQFT, the authors of \cite{Quantumreps} show that equality in part (i) of Question \ref{Mainquestion} does not always hold. However, such covers are not explicitly constructed 
in~\cite{Quantumreps}.\\

Denote by $Mod(\Sigma)$ the mapping class group of a surface $\Sigma$. Unlike the surface $S$, $\Sigma$ is not assumed to have empty boundary. When
$\Sigma$ has $p$ punctures, genus $g$ and $n$ boundary components, we will
sometimes denote it by $\Sigma^{p}_{g,n}$. \\

We now explain a relation between a rational coefficient version of 
Question~\ref{Mainquestion} 
\begin{equation}
\label{eq.QQ}
sc_{p}(H_{1}(\tilde{S};\mathbb{Z}))\otimes_{\mathbb{Z}}
\mathbb{Q} =H_1(\tilde{S},\mathbb{Q})
\end{equation}
and the Ivanov Conjecture~\cite{Ivanov}. The latter states that
$H_1(\Gamma,\mathbb{Q})=0$ for any finite index subgroup $\Gamma$ of 
$Mod(\Sigma^{p}_{g,n})$ with $g\geq 3$. \\

Fix a finite-sheeted regular cover $p\co \tilde{\Sigma} \to \Sigma$ of 
$\Sigma$, and let $\Gamma_p$ denote the subgroup of $Mod(\Sigma)$ generated
by all $\phi \in Mod(\Sigma)$ such that $\phi$ lifts to $\tilde{\phi}
\in Mod(\tilde \Sigma)$. Let $\tilde \Gamma_p$ denote the subgroup of
$Mod(\tilde \Sigma)$ generated by $\tilde \phi$ as before. \\

Boggi-Looijenga~\cite{BLconversation} observe that when~\eqref{eq.QQ} holds,
the $\tilde\Gamma$ invariant submodule of $H_{1}(\tilde{\Sigma};\mathbb{Q})$ 
is trivial. If that happens for all finite regular covers $p$ of $\Sigma$,
then Theorem C of \cite{Putman2} implies Ivanov's Conjecture for mapping 
class groups of surfaces with genus $g+1$, $n-1$ boundary components and 
$p$ punctures. A further discussion to the background of this question with 
integer and rational coefficients can be found in Section 8 of \cite{FH}.\\

For surfaces with nonempty boundary or at least one puncture, 
Farb-Hensel~\cite{FH} provide a representation theoretic framework 
in which to study part (ii) of Question \ref{Mainquestion} with integer
or rational coefficients. For punctured surfaces, it was shown in a 
recent paper of Malestein-Putman~\cite{MalesteinPutman} that 
Question~\ref{Mainquestion} fails. \\

Our main theorem is an explicit family of counterexamples to 
Question~\ref{Mainquestion}. These are iterated homology surface coverings,
with at least two iterations, where the last one uses an integer $m \geq 3$.

\begin{thm}
\label{thm.1}
For the examples of Section~\ref{sub:non_example}, equality in 
part (i) of Question~\ref{Mainquestion} is false.
\end{thm} 

The intuition that lifts of simple curves should generate homology possibly stems from the fact that counterexamples should be expected to have large genus; ``small'' genus coverings disproportionately satisfy a plethora of conditions that guarantee this. For example, it seems to be well-known that when the deck transformation group is Abelian, $sc_{p}(H_{1}(\tilde{S};\mathbb{Z})) =H_1(\tilde{S},\mathbb{Z})$. (As was explained to the author by Marco Boggi, \cite{Boggiemail}, this claim follows from arguments of Boggi-Looijenga and \cite{Looijenga} with rational and hence integral coefficients. It was proven directly for integral coefficients in an earlier version of this paper. For fundamental groups of
surfaces with nonempty boundary or at least one puncture and 
complex coefficients, this claim is Proposition 3.1 of \cite{FH}.) \\

\textbf{Organisation of Paper.} Section~\ref{sect:background} recalls and
provides some background and useful notation. Section~\ref{sect:example} 
studies a family of covering spaces in detail. These covering spaces are compositions of well known covering spaces, and the author makes no claims of originality in this section. The properties of the covering spaces that will be needed are simple and elementary, and hence are proven directly for completeness. The same is true for Subsection \ref{relations}, in which it is explained how to obtain spanning sets for homology of covers using relations in the deck transformation group. These results and examples are used in Subsection \ref{rationalvsintegral} to highlight differences between integral and rational homology, and in Subsection \ref{sub:non_example} to construct examples for which $sc_{p}(H_{1}(\tilde{S};\mathbb{Z}))$ is a proper submodule of $H_{1}(\tilde{S};\mathbb{Z})$. 

\subsection*{Acknowledgments} 

As noted above, the author became aware of this question on MathOverflow, and 
is grateful to Juli\'{e}n March\'{e} for posting the question, and the 
subsequent discussion by Richard Kent and Ian Agol. Mustafa Korkmaz and 
Sebastian Hensel pointed out an error in the first formulation of this paper. 
This paper was greatly improved as a result of communication with Marco Boggi,
Stavros Garoufalidis, Neil Hoffman, Thomas Koberda, Eduard Looijenga, 
Andrew Putman, and the detailed comments of the anonymous referee. 
The author is also grateful to Ferruh \"Ozbudak for a fascinating discussion 
on similar methods in coding theory. This research was funded by a 
T\"{u}bitak Research Fellowship 2216 and thanks the University of Melbourne 
for its hospitality during the initial and final stages of this project. 


\section{Assumptions and background}
\label{sect:background}

Since the 2-sphere is simply connected, hence has no nontrivial covers,
Question~\ref{Mainquestion} holds for genus 0 surfaces. Moreover, 
as pointed out by Ian Agol in \cite{MO}, part (i) of 
Question~\ref{Mainquestion} holds for genus 1 surfaces.
Therefore we only need consider closed surfaces $S$ of genus at least two. 
\\

\textbf{Curves and intersection numbers.} By a \emph{curve} in $S$ we mean
the free homotopy class of the image of a smooth map (which can be
taken to be an immersion) of $S^1$ into $S$. In this convention, a curve is necessarily closed and connected. A curve $\gamma$ is said to be \textit{simple} if the free homotopy class contains an embedding of $S^1$ in $S$. \\

When it is necessary to work with based curves, the assumption will often be made that wherever necessary, a representative of the free homotopy class is conjugated by an arc, to obtain a curve passing through the base point. The counterexamples constructed in this paper are obtained by iterated Abelian covers. When analysing how a curve lifts, it will therefore only be necessary to know the homology class of the curve in the intermediate cover in question.\\

If $c$ is a curve in a surface $S$, then $[c] \in H_1(S,\mathbb{Z})$ will
denote the corresponding homology class. \\

\textbf{$d$-lifts.} Given a finite sheeted regular covering $p\co \tilde{S} \to S$, the deck transformation group is denoted by $D$. For an
element $\gamma \in \pi_1(S)$, let $d=d(\gamma)$ denote the smallest natural
number for which $\gamma^d \in \pi_1(\tilde{S}) \subset \pi_1(S)$. Note that
$d$ exists and $d \leq |D|$, where $|D|$ is the number of elements of $D$. 
In that case, we will say that $\gamma$ $d$-lifts. \\

\textbf{Primitivity. }A homology class $h$ is \textit{primitive} if it is nontrivial and there does not exist an integer $k>1$ and a homology class $h_{prim}$ such that $h=kh_{prim}$. \\


\section{Homology coverings of a surface}
\label{sect:example}

In this section we recall the definition of a homology covering
of a surface and its basic properties. Our goal is to show that 
iterated homology coverings give a counterexample to part (i) of
Question~\ref{Mainquestion}. 

\subsection{Definition of a homology covering}
\label{sub.homcover}

We begin by recalling the well-known homology covers of a closed genus
$g$ surface $S$. Fix a natural number $m$ and let $p: \tilde \Sigma \to S$
denote the covering space of $S$ corresponding to the epimorphism
$\phi: \pi_1(S) \to H_1(S,\mathbb{Z}/m \mathbb{Z})$ is given by the composition
\begin{equation}
\label{eq.phi}
\pi_1(S) \to H_1(S,\mathbb{Z}) \to H_1(S,\mathbb{Z}/m \mathbb{Z})
\end{equation}
of the Hurewitz homomorphism with the reduction of homology modulo $m$. Let 
$D \simeq (\mathbb{Z}/m \mathbb{Z})^{2g}$ denote the
deck transformation group of $\phi$. These coverings are known as the 
\textit{mod-$m$-homology coverings} of $S$. 
Using an Euler characteristic argument, it follows that the genus of 
$\tilde{S}$ is $m^{2g}(g-1)+1$. Homology coverings are characteristic in
the following sense: the subgroup $\pi_1(\tilde S)$ of $\pi_1(S)$ 
is invariant under all surface diffeomorphisms of $S$. \\

\subsection{Homology coverings with $g=2$}
\label{sub.g=2}

When $g=2$, we can give an explicit description of the covering $\tilde S$
as follows. Write $S=t_1 \cup t_2$ where $t_i$ for $i=1,2$ are genus 
1 subsurfaces of $S$ with one boundary component. The pre-image of each $t_i$
under $\phi$ consists of $m^2$ copies of an $m^{2}$-holed torus, each
of which is an $m^{2}$-fold cover of $t_i$, as
illustrated in Figure \ref{holytorus}. \\

\begin{figure}[!htpb]
\centering
\includegraphics[height=0.15\textheight]{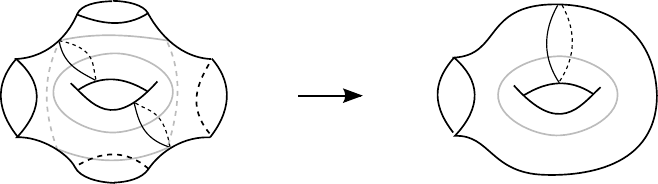}
\caption{
The pre-image of a 1-holed torus under the covering when $m=2$.}
\label{holytorus}
\end{figure}

The pre-images $\phi^{-1}(t_i)$ for $i=1,2$ are glued together as follows.
Let $K_{m^2,m^2}$ be the bi-partite graph with vertices the connected 
components of $\phi^{-1}(t_i)$ for $i=1,2$. Each connected component of
$\phi^{-1}(t_1)$ is glued to each connected component of $\phi^{-1}(t_2)$ along
a boundary curve, and this is represented by an edge of $K_{m^2,m^2}$. 
This is illustrated in Figure \ref{cube} for $m=2$, where the graph 
$K_{m^{2},m^{2}}$ is shown in grey. \\

\begin{figure}[!htpb]
\centering
\includegraphics[height=0.30\textheight]{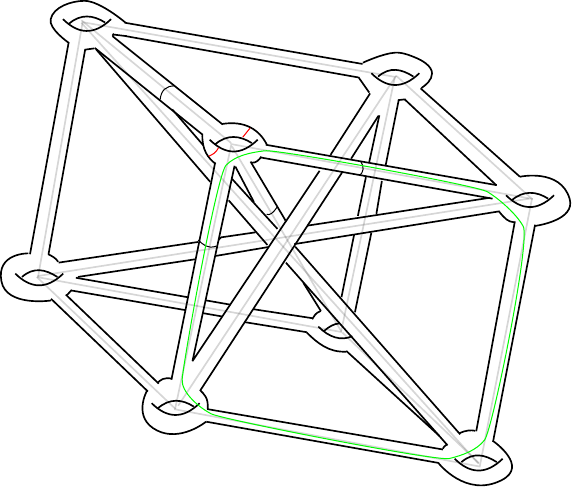}
\caption{The covering space $\tilde{S}$ for $m=2$. 
The red curves are some connected components of pre-images of the 
generator $a_1$ of $\pi_{1}(S)$. The green curve is a connected component 
of the lift of $b_{1}b_{2}$, and the black curves are connected components 
of the lift of $[a_{1},b_{1}]$.}
\label{cube}
\end{figure}

This subsection is now concluded with a useful lemma.

\begin{lem}
\label{nullnonsimple}
Let $\tilde{S}\rightarrow S$ be the homology cover of~\eqref{eq.phi} with
$g=2$. 
All simple, null homologous curves in $S$ 1-lift to nonseparating curves in 
$\tilde{S}$. Moreover, no curve in a primitive homology class in $S$ 
1-lifts. 
\end{lem}

\begin{proof}
To start off with, the fact that null homologous curves 1-lift is a consequence of the fact that the cover has Abelian deck transformation group. \\

In Figure \ref{cube}, the black curves are lifts of simple null homologous curves from $S$. These black curves all 1-lift to non-separating curves in $\tilde{S}$. The covering is characteristic,
so this observation is true independently of the choice of basis $\{a_{1}, b_{1},a_{2}, b_{2}\}$ from Equation \eqref{presentations}. It follows that all simple null homologous curves 1-lift to simple, non-separating curves in $\tilde{S}$. An analogous argument shows this is also true for $m>3$.
\end{proof}

\subsection{The homology of a homology covering}
\label{relations}

In this subsection, we will assume that $S$ is a closed surface of genus 2.
It will be useful to describe the homology covers of a surface $S$
using a fixed presentation 
\begin{equation}
\label{presentations}
\pi_1(S) = \langle a_1, b_1, a_{1}, b_{2} \,\, | \,\,
[a_{1},b_{1}] [a_{2},b_{2}]\rangle
\end{equation}
where $a_i,b_i$ are curves representing a usual symplectic 
basis for $H_{1}(S;\mathbb{Z})$, satisfying 
$i(a_{i}, a_{j})=i(b_{i}, b_{j})=0$ and $i(a_{i},b_{j})=\delta^{i}_{j}$.\\

The homomorphism~\eqref{eq.phi} is given explicitly by
\begin{equation}\label{thecover}
\begin{split}
a_{1}\mapsto (1,0,0,0)\\
b_{1}\mapsto  (0,1,0,0)\\
a_{2}\mapsto (0,0,1,0)\\
b_{2}\mapsto  (0,0,0,1)
\end{split}
\end{equation}

It will now be shown how to use the relations of the deck transformation group to obtain a generating set for $H_{1}(\tilde{S},\mathbb{Z})$. This subsection is only needed to show the necessity of the assumption $m\geq 3$ in the construction later on.  \\

Suppose now that $D$ is any group for which there is the short exact sequence
\begin{equation}
\label{phi}
1\rightarrow \pi_{1}(\tilde{S}) \rightarrow \pi_{1}(S) 
\xrightarrow{\phi} D \rightarrow 1
\end{equation}
for some covering space $\tilde{S}$. Let $\{g_{1}, \ldots, g_{n}\}$ be a set of elements of $\pi_{1}(S)$ whose image under $\phi$ is a generating set for $D$. Let $r=w(g_{1}, \ldots, g_{n})$ be a word in the elements $\{g_{1}, \ldots, g_{n}\}$ that is mapped to the identity by $\phi$, i.e. $w(\phi(g_{1}), \ldots, \phi(g_{n}))=I\in D$. The word $r$ could be either nontrivial in $\pi_{1}(S)$, or it could be a product of conjugates of the relation $[a_{1},b_{1}][a_{2},b_{2}]$ in $\pi_{1}(S)$.\\

Note that $\{\phi(a_{1}), \phi(a_{2}), \phi(b_{1}), \phi(b_{2})\}$ is a generating set for $D$, where $\{a_{1}, a_{2}, b_{1}, b_{2}\}$ is the choice of generating set from Equation \eqref{presentations}. Due to the assumption that the genus of $S$ is two, $D$ can always be generated by four generators.\\

If $r=r(g_1,\ldots,g_n)$ is a word in $g_1,\ldots,g_n$, then 
$\phi(r)=r(\phi(g_1),\ldots,\phi(g_n))$.
A set of words $\{r_{1}(g_{1}, \ldots, g_{n}), 
r_{2}(g_{1}, \ldots, g_{n}), \ldots, r_{k}(g_{1}, \ldots, g_{n})\}$ is 
a complete set of relations for $D$ if 
\begin{equation*}
\{\phi(g_{1}), \ldots, \phi(g_{n})\, \mid \, \phi(r_1),\ldots, \phi(r_k)\}
\end{equation*}
is a presentation for $D$.\\

The next lemma is a corollary of a presumably well known group theoretic statement. As the author could not find a reference, a proof is given here for the sake of completeness.

\begin{lem}
\label{generatingset}
Suppose $D$ can be generated by no fewer than four generators. Let $\{r_{1}, r_{2}, \ldots, r_{k}\}$ be a set of words in $\pi_{1}(S)$ mapping to a complete set of relations for $D$. Then the set of homology classes of connected components of the pre-images of the curves representing the words $r_{1}, r_{2}, \ldots, r_{k}$ is a generating set for $H_{1}(\tilde{S};\mathbb{Z})$.
\end{lem}

\begin{proof}
Consider a presentation for $D$ given by
\begin{equation*}
\{\phi(a_{1}), \phi(a_{2}), \phi(b_{1}), \phi(b_{2}) \, \mid \,
\phi(r_{1}^{'}), \phi(r_{2}^{'}), \ldots, \phi(r_{k}^{'})\}
\end{equation*}
From the exact sequence \eqref{phi}, we see that each of the $r_{i}^{'}$ represents an element of $\pi_{1}(\tilde{S})$.\\

When the image of $\{r_{1}^{'}, r_{2}^{'}, \ldots, r_{k}^{'}\}$ under $\phi$ is a complete set of relations for $D$, it follows that any element of $\pi_{1}(\tilde{S})$ is a product of conjugates of elements of the set $\{r_{1}, r_{2}, \ldots, r_{k}\}$. Let $c$ be a loop representing the element $r_{i}^{'}$. The connected components of $p^{-1}(c)$ correspond to conjugates of $r_{i}^{'}$. Therefore, the connected components of the pre-images of the closed curves represented by the words $\{r_{1}^{'}, r_{2}^{'}, \ldots, r_{k}^{'}\}$ are a generating set for $H_{1}(\tilde{S},\mathbb{Z})$.\\

We now use the assumption that $D$ has no fewer than four generators to show that this is true for any presentation of $D$. Another presentation for $D$ can be written as follows
\begin{equation*}
\{w_{1}, w_{2}, w_{3}, w_{4}\, \mid \, \phi(r_{1}), \phi(r_{2}), \ldots, 
\phi(r_{m})\}
\end{equation*}
where $w_{1}, w_{2}, w_{3}$ and $w_{4}$ are words in $\phi(a_{1})$, $\phi(a_{2})$, $\phi(b_{1})$ and $\phi(b_{2})$, and $r_{1}, r_{2}, \ldots, r_{m}$ are products of conjugates of $r_{1}^{'}, r_{2}^{'}, \ldots, r_{k}^{'}$. For the same reason as before, the connected components of the pre-images of the closed curves in $S$ represented by the words $\{r_{1}, r_{2}, \ldots, r_{m}\}$ are a generating set for $H_{1}(\tilde{S};\mathbb{Z})$. 
\end{proof}

\begin{rem}
Note that the last sentence of the proof is not necessarily true for a proper subset of $\{r_{1}, r_{2}, \ldots, r_{m}\}$. If one or more of $w_{1}, w_{2}, w_{3}$ or $w_{4}$ is mapped to the identity in $D$, there is a presentation for $D$ with generators consisting of a proper subset of $\{w_{1}, w_{2}, w_{3}, w_{4}\}$ and relations consisting of a proper subset of $\{\phi(r_{1}), \phi(r_{2}), \ldots, \phi(r_{m})\}$. However, the assumption that $D$ is generated by no fewer than four elements rules out the possibility of a presentation for $D$ with relations consisting of a proper subset of $\{\phi(r_{1}), \phi(r_{2}), \ldots, \phi(r_{m})\}$.
\end{rem}

To use Lemma \ref{generatingset}, a complete set of relations for $D$ is needed. To start off with, there are the relations $\phi^{m}(a_{i})=1$ and $\phi^{m}(b_{i})=1$. These relations correspond to the submodule of $H_{1}(\tilde{S};\mathbb{Z})$ spanned by connected components of pre-images of the generators. In Figure \ref{cube} with $m=2$, some examples are drawn in red. Other relations are, for example, commutation relations or the relations stating that the remaining group elements have order $m$ in the deck transformation group. When $m=2$, the commutation relations are a consequence of the relations stating that all 16 group elements have order $m$. For example,
\begin{align*}
\phi(a_{1})\phi(b_{1})
&=(\phi(a_{1})\phi(b_{1}))^{-1} \qquad
\text{ since }\phi(a_{1})\phi(b_{1}) \text{ has order 2}\\
&=\phi(b_{1})^{-1}\phi(a_{1})^{-1}\\
&=\phi(b_{1})\phi(a_{1}) \qquad\quad\,\,\,\, 
\text{ since }\phi(a_{1})\text{ and }\phi(b_{1})\text{ each have order 2}
\end{align*}

It will be shown later that this is a peculiarity of $m=2$; as shown in Lemma \ref{secondthought}, for $m>2$, we also need commutation relations. For $m=2$, the relations stating that all elements of the deck tranformation group are of order two are a complete set of relations for $D$. By Lemma \ref{generatingset}, this gives us a set of simple, nonseparating curves whose pre-images span $H_{1}(\tilde{S};\mathbb{Z})$. \\

\subsection{Integral versus rational homology}
\label{rationalvsintegral}

Examples for which $sc_{p}(H_{1}(\tilde{S};\mathbb{Z}))$ can not be all of $H_{1}(\tilde{S};\mathbb{Z})$ will now be constructed. This is done by showing that for $m>2$, connected components of pre-images of simple, nonseparating curves do not span $H_{1}(\tilde{S};\mathbb{Z})$; if we want to span $H_{1}(\tilde{S};\mathbb{Z})$ with connected components of pre-images of simple curves, separating curves will also be needed. The promised examples are then obtained by taking the composition of two such covering spaces, using Lemma \ref{nullnonsimple}.

\begin{lem}
In the homology covering space $p:\tilde{S}\rightarrow S$ of~\eqref{eq.phi} 
with $m\geq3$ and $g \geq 2$, connected components of pre-images of simple, 
nonseparating curves do not span $H_{1}(\tilde{S};\mathbb{Z})$.
\label{secondthought}
\end{lem}

\begin{proof}
The lemma will be proven for $m=3$ and it is claimed that analogous arguments work for $m>3$.\\

We will show by contradiction 
that $[p^{-1}([a_{1},b_{1}])] \in H_1(\tilde S,\mathbb{Z})$
is not in the span of homology classes of pre-images of simple, 
nonseparating curves of $S$. \\

Suppose a connected component of $p^{-1}([a_{1},b_{1}])$ is in the span of connected components of pre-images of simple, nonseparating curves. 
In the group $\pi_{1}(S)$ there is therefore the relation 
\begin{equation}
\label{relation}
[a_{1},b_{1}]^{-1}\gamma_{1}^{3}\gamma_{2}^{3}\gamma_{3}^{3}\ldots\gamma_{k}^{3}
\kappa=1 \in \pi_1(\tilde S)
\end{equation}
where $\kappa$ is in the subgroup $[\pi_{1}(\tilde{S}), \pi_{1}(\tilde{S})]$ of $\pi_{1}(S)$ and $\gamma_i$ are elements of $\pi_1(S)$ representing 
simple closed curves in $S$. \\

Let $N$ denote the subgroup of $\pi_{1}(S)$ normally generated by words 
of Equation~\eqref{cubesandcomms},
\begin{equation}
\label{cubesandcomms}
[w,a_{i}^{\pm3}]\text{, }[w,b_{i}^{\pm3}]
\text{ for }i\in 1,\ldots,g, \,\, w \in \pi_1(S) \qquad \text{and} \qquad
[w,[\pi_{1}(S),\pi_{1}(S)]], \,\, w \in \pi_1(S) \,.
\end{equation}
and let $N_{2}$ be the quotient 
\begin{equation*}
\frac{\pi_{1}(S)}{[\pi_{1}(S),[\pi_{1}(S), \pi_{1}(S)]]}
\end{equation*}

\begin{claim}
\label{symbolmanipulation}
$\gamma_{1}^{3}\gamma_{2}^{3}\gamma_{3}^{3}\ldots\gamma_{k}^{3} \in N$.
\end{claim}

\begin{proof}(of the Claim)
The product $\gamma_{1}^{3}\gamma_{2}^{3}\gamma_{3}^{3}\ldots\gamma_{k}^{3}$ is null homologous in $S$, so for every generator $a_{i}$ (respectively $b_i$), $i\in \{1,2\}$, the sum of the powers in the product $\gamma_{1}^{3}\gamma_{2}^{3}\gamma_{3}^{3}\ldots\gamma_{k}^{3}$ must be zero. This implies that the elements of the set $\{\gamma_{h}\}$ can not all be elements of the generating set $\{a_{1}, a_{2}, b_{1}, b_{2}\}$. Assume otherwise: then for every $a^{3}_{i}$ $i\in \{1,2\}$, (respectively $b^{3}_{i}$) there must be an $a^{-3}_i$ (respectively $b^{3}_{i}$), in which case $\gamma_{1}^{3}\gamma_{2}^{3}\gamma_{3}^{3}\ldots\gamma_{k}^{3}$ would be in the commutator subgroup of $\pi_{1}(\tilde{S})$. Since Lemma \ref{nullnonsimple} states that a connected component of $p^{-1}([a_{1},b_{1}])$ is nonseparating, this would contradict Equation \eqref{relation}.\\

Suppose 
\begin{equation*}
\gamma_{h}=x_{1}x_{2}\ldots x_{n}:=x_{1}y\text{, where each }x_{j}\in \{a_{1}, a_{2}, b_{1}, b_{2}\}\text{ for }j\in \{1,2,\ldots,n\}.
\end{equation*}

It will now be shown that $\gamma_{h}^{3}$ can be written as a product of cubes of generators, followed by an element of $N$.\\

It follows from the commutator identities $[x,yz]=[x,y][x,z]^{y}$ and $[zx,y]=[x,y]^{z}[z,y]$ that elements of $[\pi_{1}(\tilde{S}),\pi_{1}(\tilde{S})]$ are normally generated by words from Equation \eqref{cubesandcomms}, hence $[\pi_{1}(\tilde{S}), \pi_{1}(\tilde{S}]\in N$.\\

Note that 
\begin{align*}
&\gamma_{h}^{3}=x_{1}yx_{1}yx_{1}y\\
&=x_{1}yx_{1}y^{2}x_{1}[x_{1}^{-1},y^{-1}]\text{, using }x_{1}y=yx_{1}[x_{1}^{-1},y^{-1}]\\
&=x_{1}^{2}y[y^{-1},x_{1}^{-1}]y^{2}x_{1}[x_{1}^{-1},y^{-1}]
\end{align*}
In $N_{2}$, this equals
\begin{align*}
&x_{1}^{2}y^{3}x_{1}[y^{-1}, x_{1}^{-1}][x_{1}^{-1}, y^{-1}]\\
&=x_{1}^{2}y^{3}x_{1}\\
&=x_{1}^{3}y^{3}[y^{-3},x_{1}^{-1}]
\end{align*}

In other words, 
\begin{equation}
\label{moresymbols}
(x_{1}y)^{3}=x_{1}^{3}y^{3}[y^{-3}, x_{1}^{-1}]n
\end{equation}
where $n\in N$.\\

It follows from Equation \ref{moresymbols} and the commutator identity $[zx,y]=[x,y]^{z}[z,y]$ that $[y^{-3}, x_{1}^{-1}]\in N$. This implies that $\gamma_{h}^{3}=x_{1}^{3}y^{3}n_{1}$, where $n_{1}\in N$.\\

Note that rearranging the orders of cubes of generators and elements of $N$ only introduces more elements in $N$. Repeating the argument just given on the shorter word $y$, and rearranging, shows that $\gamma_{h}^{3}$ is a product of cubes of generators, followed by an element of $N$.\\


It is therefore possible to write $\gamma_{1}^{3}\gamma_{2}^{3}\gamma_{3}^{3}\ldots\gamma_{k}^{3}$ as a product of cubes of generators, and elements of $N$. Again, rearranging the orders of cubes of generators and elements of $N$ only introduces more elements of $N$. It follows that $\gamma_{1}^{3}\gamma_{2}^{3}\gamma_{3}^{3}\ldots\gamma_{k}^{3}$ is a product of cubes of generators, and an element of $N$. Also, the products of the cubes of generators must be in $\pi_{1}(\tilde{S})$, by construction. In fact, the product of cubes must be in $[\pi_{1}(\tilde{S}), \pi_{1}(\tilde{S})]$ because Equation \eqref{relation} implies the sum of the powers of any generator in the product must be zero, otherwise $\gamma_{1}^{3}\gamma_{2}^{3}\gamma_{3}^{3}\ldots\gamma_{k}^{3}$ could not be null homologous in $S$. Since $[\pi_{1}(\tilde{S}), \pi_{1}(\tilde{S})]$ is contained in $N$, the claim follows.
\end{proof}

A contradiction will now be obtained by showing that $[a_{1}, b_{1}]$ cannot be contained in $N$.\\

Let $\psi$ be the homomorphism taking an element of $\pi_{1}(S)$ to its coset in $\pi_{1}(S)/N$, and $\psi_1$ be the homomorphism taking an element of $\pi_{1}(S)$ to its coset in $\pi_{1}(S)/[\pi_{1}(S),[\pi_{1}(S),\pi_{1}(S)]]$. We now compute the image of $[\pi_{1}(S), \pi_{1}(S)]$ under $\psi$.\\

It follows from the commutator identities $[x,yz]=[x,y][x,z]^{y}$ and $[zx,y]=[x,y]^{z}[z,y]$ that $[\pi_{1}(S),\pi_{1}(S)]$ is generated by conjugates of commutators of generators. Since $[\pi_{1}(S), \pi_{1}(S)]$ maps to a subgroup in the center of the image of $\psi_{1}$, it follows that $\psi_{1}([\pi_{1}(S), \pi_{1}(S)])$ is generated by the image of commutators of generators of $\pi_{1}(S)$. The group $\psi_{1}([\pi_{1}(S), \pi_{1}(S)])$ is therefore a finitely generated, Abelian group. Note that $[\pi_{1}(S), \pi_{1}(S)]$ is a free group, and $\psi_{1}([\pi_{1}(S), \pi_{1}(S)])$ is its abelianisation, so $\psi_{1}([\pi_{1}(S), \pi_{1}(S)])$ can not be the trivial group.\\

Again using the commutator identities, it follows that $\psi_{1}([w,a_{i}^{3}])=\psi_{1}([w,a_{i}]^{3})$. Similarly for $\psi_{1}([w,b_{i}^{3}])$. If the word $w$ is not a generator, it follows from the commutator identities and the fact that the image of the commutator subgroup under $\psi_1$ is in the center, that $\psi_{1}([w,a_{i}]^{3})$ can be written as a product of cubes of commutators of generators. Since the word $w$ can be taken to be any of the four generators, it follows that the image of $[\pi_{1}(S),\pi_{1}(S)]$ under $\psi$ is a finitely generated Abelian group, each element of which has order three. In particular, $\psi([a_{1},b_{1}])$ is not the identity. This proves the promised contradiction from which the lemma follows.
\end{proof}

\begin{rem}
The argument in Lemma \ref{secondthought} does not work when $m=2$. This is because in this case we do not get any commutators of commutators in the expression for $A$, hence there is no contradiction to the existence of Equation \eqref{relation}.
\end{rem}

It follows from arguments of Boggi-Looijenga, \cite{Boggiemail} and \cite{Looijenga}, that when $D$ is Abelian, $H_{1}(\tilde{S};\mathbb{Q})$ is generated by homology classes of lifts of simple, nonseparating curves. In Figure \ref{holytorus}, for example, it is not hard to see that pairs of connected components of pre-images of a simple null homologous curve $n$ are in the span of connected components of pre-images of generators. When $m>3$, $m[\tilde{n}]$ is in the integral span of homology classes of connected components of lifts of simple, nonseparating curves. Hence $[\tilde{n}]$ is in the rational span, but not, as shown in Lemma \ref{secondthought}, in the integral span.

\subsection{Proof of Theorem~\ref{thm.1}}
\label{sub:non_example}
The promised families of examples for which lifts of simple curves do not span the integral homology of the covering space will now be constructed. The
examples are iterated homology coverings of $S$, with at least two
iterations, where the last homology covering uses an integer $m \geq 3$.\\

Let $\tilde{S}\rightarrow S$ be the covering with $m\geq2$ just studied. Repeat the same construction, only with larger genus, and $m>2$, on $\tilde{S}$ to obtain a cover $\tilde{\tilde{S}}\rightarrow S$ factoring through $\tilde{S}$. That the result is a regular cover follows from the fact that it is a composition of two characteristic covers.\\

It is possible to see almost immediately that $sc_{p}(H_{1}(\tilde{\tilde{S}},\mathbb{Z}))$ can not be all of $H_{1}(\tilde{\tilde{S}},\mathbb{Z})$. In Lemma \ref{secondthought} we saw that 1-lifts of simple null homologous curves from $\tilde{S}$ were needed to generate $H_{1}(\tilde{\tilde{S}}; \mathbb{Z})$. However, by Lemma \ref{nullnonsimple}, no simple null homologous curves in $\tilde{S}$ project onto simple curves in $S$.\\

\bibliography{scpbib}

\begin{thebibliography}{10}

\bibitem{Boggiemail}
M.~Boggi.
\newblock Personal communication.
\newblock 2015.

\bibitem{FH}
B.~Farb and S.~Hensel.
\newblock Finite covers of graphs, their primitive homology, and representation
  theory.
\newblock {\em New York J. Math.}, 22:1365--1391, 2016.

\bibitem{Marche}
L.~Funar and J.~March\'e.
\newblock The first {J}ohnson subgroups act ergodically on {$\rm
  SU_2$}-character varieties.
\newblock {\em J. Differential Geom.}, 95(3):407--418, 2013.

\bibitem{Ivanov}
N.~Ivanov.
\newblock Fifteen problems about the mapping class groups.
\newblock In {\em Problems on mapping class groups and related topics},
  volume~74 of {\em Proc. Sympos. Pure Math.}, pages 71--80. Amer. Math. Soc.,
  Providence, RI, 2006.

\bibitem{Johnson}
D.~Johnson.
\newblock A survey of the {T}orelli group.
\newblock In {\em Low-dimensional topology ({S}an {F}rancisco, {C}alif.,
  1981)}, volume~20 of {\em Contemp. Math.}, pages 165--179. Amer. Math. Soc.,
  Providence, RI, 1983.

\bibitem{Quantumreps}
T.~Koberda and R.~Santharoubane.
\newblock Quotients of surfaces groups and homology of finite covers via
  quantum representations.
\newblock {\em Inventiones Mathematicae}, 206:269--292, 2016.

\bibitem{Looijenga}
E.~Looijenga.
\newblock Prym representations of mapping class groups.
\newblock {\em Geom. Dedicata}, 64(1):69--83, 1997.

\bibitem{BLconversation}
E.~Looijenga.
\newblock Personal communication.
\newblock 2015.

\bibitem{MalesteinPutman}
J.~Malestein and A.~Putman.
\newblock Simple closed curves, finite covers of surfaces, and power subgroups
  of {O}ut(${F}_n$), 2017.

\bibitem{MO}
J.~March\'{e}.
\newblock Question on mathoverflow.
\newblock
  \url{http://mathoverflow.net/questions/86894/homology-generated-by-lifts-of-%
simple-curves}, (accessed 6-Aug-2015).

\bibitem{Putman2}
A.~Putman and B.~Wieland.
\newblock Abelian quotients of subgroups of the mappings class group and higher
  {P}rym representations.
\newblock {\em J. Lond. Math. Soc. (2)}, 88(1):79--96, 2013.

\end{thebibliography}
\bibliographystyle{plain}
\end{document}